\documentclass[11pt]{article}

\usepackage{amsfonts,amsmath,latexsym,color,epsfig,hyperref,enumitem, amssymb,bbm}

\newtheorem{theorem}{Theorem}[section]
\newtheorem{lemma}[theorem]{Lemma}

\newtheorem{cor}[theorem]{Corollary}
\newtheorem{prop}[theorem]{Proposition}

\renewcommand{\le}{\leqslant}
\renewcommand{\ge}{\geqslant}

\newcommand{\R}{\mathbb R}

\newenvironment{proof}
      {\medskip\noindent{\bf Proof:}\hspace{1mm}}
      {\hfill$\Box$\medskip}
\def\qed{\ifvmode\mbox{ }\else\unskip\fi\hskip 1em plus 10fill$\Box$}

\input{epsf}

\makeatletter
\def\Ddots{\mathinner{\mkern1mu\raise\p@
\vbox{\kern7\p@\hbox{.}}\mkern2mu
\raise4\p@\hbox{.}\mkern2mu\raise7\p@\hbox{.}\mkern1mu}}
\makeatother

\def\R{\mathbb R}

\title{\vspace{-1.3cm} Upper bounds for Heilbronn's triangle problem in higher dimensions}
\author{Dmitrii Zakharov\thanks{Department of Mathematics, Massachusetts Institute of Technology, 77 Massachusetts Avenue, Cambridge, MA 02139, USA, Email:~\href{mailto:zakhdm@mit.edu}{\tt zakhdm@mit.edu}.}}

\date{}

\begin{document}

\maketitle

\vspace{-0.5cm}

\abstract{
We develop a new simple approach to prove upper bounds for generalizations of the Heilbronn's triangle problem in higher dimensions. Among other things, we show the following: for fixed $d \ge 1$, any subset of $[0, 1]^d$ of size $n$ contains
\begin{itemize}
    \item $d+1$ points which span a simplex of volume at most $C_d n^{-\log d+ 6}$,
    \item $1.1 d$ points whose convex hull has volume at most $C_d n^{-1.1}$,
    \item $k\ge 4\sqrt{d}$ points which span a $(k-1)$-dimensional simplex of volume at most $C_d n^{-\frac{k-1}{d} - \frac{k^2}{8d^2}}$.
\end{itemize}

}

\section{Introduction}

The Heilbronn's triangle problem asks for the smallest number $\Delta = \Delta(n)$ such that among any $n$ points in the unit square $[0, 1]^2$ one can find 3 which form a triangle of area at most $\Delta$. After a series of developments by Roth \cite{R1}, \cite{R2}, \cite{R3} and Schmidt \cite{S}, in 1981-82 Komlos--Pintz--Szemer\'edi \cite{KPS1}, \cite{KPS2} obtained the following bounds:
\begin{equation}\label{plane}
n^{-2} \log n \lesssim \Delta(n) \lesssim n^{-\frac{8}{7}+\varepsilon}
\end{equation}
where $\varepsilon > 0$ can be taken arbitrarily small. Very recently, Cohen, Pohoata and the author \cite{CPZ} improved the upper bound in (\ref{plane}) to $\Delta(n) \ll n^{-\frac{8}{7}-\frac{1}{2000}}$.

In this note, we consider a higher dimensional analogue of this question, namely, given $1 \le k \le d+1$ define $\Delta_{k, d}(n)$ to be the minimum $\Delta > 0$ such that among any $n$ points in $[0, 1]^d$ there are $k$ points which form a simplex with $(k-1)$-dimensional volume at most $\Delta$. 
Setting $d=2, ~k=3$ recovers the original Heilbronn's triangle problem. 

Simple probabilistic and packing arguments imply that 
\begin{equation}\label{ez}
n^{- \frac{k-1}{d-k+2} } \lesssim \Delta_{k, d}(n) \lesssim n^{- \frac{k-1}{d}},
\end{equation}
for any fixed $1 \le k \le d+1$ (see also \cite{Bar}, \cite{BarN}). 
The lower bound can be improved by a logarithmic factor using results on independence numbers of sparse hypergraphs, see \cite{L} for details. Lefmann \cite{L} improved the upper bound by a polynomial factor for all {\em even} $k \ge 2$:
\begin{equation}\label{lef}
\Delta_{k, d}(n) \lesssim n^{- \frac{k-1}{d} - \frac{k-2}{2d(d-1)}},
\end{equation}
earlier, Brass \cite{B} proved this in the special case $k=d+1$. For $d\ge 3$ and odd $k$, no improvements over (\ref{ez}) were known and the first interesting special case is $d=3$, $k=3$.

In this note we establish the following inequality between functions $\Delta_{k, d}$:
\begin{theorem}\label{main}
Fix $2 \le k \le d+1$ and $1 \le \ell < k$. Then for some constant $C = C(d)$ and all $n' > C n$ we have:
\begin{equation}\label{eqmain}
    \Delta_{k, d}(n') \le C \Delta_{\ell, d}(n) \Delta_{k-\ell ,d-\ell}(n).
\end{equation}
\end{theorem}

Combining (\ref{eqmain}) with known upper bounds on $\Delta_{k, d}$ leads to new inequalities. 
The most notable improvement appears in the case when $k=d+1$. Recall that known estimates give $n^{-d} \lessapprox \Delta_{d+1, d}(n) \lesssim n^{-1}$ for all $d$ and $\Delta_{d+1, d}(n) \lesssim n^{-1-\frac{1}{2d}}$ when $d$ is odd. We show:

\begin{cor}\label{cor2}
We have the following bounds:
\begin{align*}
    \Delta_{4,3}(n) \lesssim n^{-\frac43},~~ \Delta_{5,4}(n) \lesssim n^{-\frac32}, ~~ \Delta_{6, 5}(n) \lesssim n^{-\frac{33}{20}},\\
    \Delta_{7, 6}(n) \lesssim n^{-1.676}, ~~ \Delta_{8, 7}(n) \lesssim n^{-1.792}, ~~ \Delta_{9, 8}(n) \lesssim n^{-\frac{19}{10}}.
\end{align*}
\end{cor}

Note that these bounds significantly improve the previously known estimates on $\Delta_{d+1, d}$. Corollary \ref{cor2} follows from a sequence of repeated applications of Theorem \ref{main} in some optimized way. One can continue this process indefinitely to write down the best possible upper bound on $\Delta_{d+1, d}(n)$ for each fixed value of $d$. The exponents quickly become rather messy so we will not write them down for $d\ge 9$. Instead, we have the following general result:

\begin{cor}\label{cor1}
For any fixed $d \ge 3$ we have $\Delta_{d+1, d}(n) \ll n^{- \log d + 6}$. 
\end{cor}

We did not attempt to optimize the constant $6$ in the exponent above. It can be done by a more careful analysis of the recursion from Theorem \ref{main} but the resulting bound is likely very far from optimal anyway. For a given $d$ it is more efficient to apply Theorem \ref{main} directly to obtain the best possible exponent. 

For arbitrary $k \le d$ we get results of the following form:

\begin{cor}\label{cor15}
For all $d\ge 3$ and $2\le k \le d+1$ we have $\Delta_{k, d}(n) \lesssim n^{-\frac{1}{2} \log \frac{d+2}{d-k+2}}$.
\end{cor}

The factor of $\frac{1}{2}$ in the bound can be improved to 1 at the expense of an additive constant error in the exponent, similar to Corollary \ref{cor2}. We decided to keep a slightly worse bound for simplicity.

Note that the bound on $\Delta_{k, d}(n)$ in Corollary \ref{cor15} improves the trivial upper bound (\ref{ez}) only if $k \ge (1-\varepsilon)d $ for a small constant $\varepsilon \approx 0.1$. For smaller values of $k$, Theorem \ref{main} gives the following improvement to (\ref{ez}):
\begin{cor}\label{cor3}
For any $d \ge 9$ and $3\sqrt{d} \le k \le d$ we have $\Delta_{k, d}(n) \lesssim n^{-\frac{k-1}{d} - \frac{k^2}{8d^2}}$.
\end{cor}

Note that if $2\le k \le (1-\varepsilon)d$, then by (\ref{ez}) we have $\Delta_{k, d}(n) \gtrsim n^{-\frac{k-1}{d-k+2}} \ge n^{-\frac{k-1}{d} - \frac{2k^2}{\varepsilon d^2}} $. So, in a way, the improvement in Corollary \ref{cor3} is optimal up to a constant factor of $16\varepsilon^{-1}$ for all $3\sqrt{d} \le k \le (1-\varepsilon)d$. 

For $k \le c \sqrt{d}$ our ideas do not lead to any improvements, so the Lefmann's bound (\ref{lef}) which holds for even $k$ is still the best known upper bound on $\Delta_{k,n}(n)$.

Another natural generalization of the Heilbronn's triangle problem is to ask for more points with convex hull of the smallest volume.
Namely, for each $k\ge d+1$, we can define a Heilbronn-type function $\Delta_{k, d}(n)$ as the minimum number $\Delta$ such that among any $n$ points in $[0, 1]^d$ there are $k$ points whose convex hull has $d$-dimensional volume at most $\Delta$. 

By a packing argument, for arbitrary $k > d \ge 2$ we have a trivial upper bound $\Delta_{k, d} \ll n^{-1}$ and no asymptotic improvements of this inequality were known prior to this work. For $d= 2$ the best lower bound is due to Lefmann \cite{L2}:
\begin{equation}\label{lef2}
\Delta_{k, 2}(n) \gtrsim (\log n)^{\frac{1}{k-2}} n^{-\frac{k-1}{k-2}}, 
\end{equation}
which is a logarithmic improvement over a bound by Bertram-Kretzberg, Hofmeister and Lefmann \cite{BHL}. A simple probabilistic argument gives the following lower bound for arbitrary $k > d \ge 2$:
\begin{equation}\label{lower}
    \Delta_{k, d}(n) \gtrsim n^{-\frac{k-1}{k-d}},
\end{equation}
see (\cite{L3}) for a proof. It seems likely that one should be able to improve (\ref{lower}) by a logarithmic factor similarly to (\ref{lef2}) but we do not pursue this direction here.

Our method allows us to improve the trivial upper bound on $\Delta_{k, d}(n)$ when $k$ is a bit larger than $d$. 

\begin{theorem}\label{largek}
    For any $d \ge 2$ and $d+3 \le k \le 2d$ we have $\Delta_{k, d}(n) \lesssim n^{- (1-\frac{1}{k-d-1}) (\log \frac{d}{k-d-1}-1.1)}$.
\end{theorem}

Note that $\Delta_{k,d}(n)$ is monotone increasing in $k$, so Theorem \ref{largek} can also be used for $k=d+2$, to say:
$$
\Delta_{d+2, d}(n) \lesssim \min_{k \ge d+3}  n^{- (1-\frac{1}{k-d}) (\log \frac{d}{k-d}-1.1)} \lesssim n^{- \log d + \log\log d + 10},
$$
where the right hand side is achieved at $k \approx d+\log d$. The $\log \log d$ term can be removed, in principle, by a more careful analysis like we do in Section \ref{sect1}. We decided to only include a simpler bound in Theorem \ref{largek} which applies to many pairs $k$ and $d$ and still gives essentially the strongest possible bounds.

Next, note that if $d$ is large enough then the bound in Theorem \ref{largek} is less than the trivial upper bound $n^{-1}$ for all $k \le 1.1 d$. That is, our result allows us to find up to $\approx 1.1 d$ points with convex hull of volume $\lesssim n^{-1-c}$ for some constant $c > 0$ in an arbitrary $n$-element set in $[0,1]^d$. The actual constant in place of $1.1$ can be improved somewhat by a more careful analysis in the proof of Theorem \ref{largek}.

It is also interesting to determine the smallest dimension $d$, in which one can show $\Delta_{d+2, d}(n)\lesssim n^{-1-c}$ for some $c>0$ (i.e. any set of $n$ points in $[0,1]^d$ contains $d+2$ points with convex hull volume at most $n^{-1-c}$). Our approach gives

\begin{prop}\label{smalld}
    We have $\Delta_{9, 7}(n) \lesssim n^{-\frac{23}{21}}$.
\end{prop}

It would be interesting to lower the dimension $d=7$ further (with the ultimate goal being showing that $\Delta_{4, 2}(n)\lesssim n^{-1-c}$).

It would be interesting to combine the approach introduced in this paper with the analytic method of Roth \cite{R3}, which was developed further by Koml\'os--Pintz--Szemer\'edi \cite{KPS2} and then by Cohen, Pohoata and the author \cite{CPZ}.
Note however that all these techniques applied to the `main' problem of estimating $\Delta_{d+1, d}$ have a natural barrier at the exponent $n^{-1-\frac12-\frac13-\ldots - \frac1d} = n^{- \log d + O(1)}$. So improving the main term in Corollary \ref{cor1} likely requires new ideas.
On the other hand, our argument seems rather wasteful for small values of $d$ so perhaps the bounds on $\Delta_{k,d}$ for small values of $d$ and $k$ can be improved further.

{\em Notation.} We use the standard asymptotic notation. If not specified, letters $c$ and $C$ usually denote constants not depending on $n$ and which might change from line to line. All logarithms are in the natural base.

\paragraph{Acknowledgements.} I thank Lisa Sauermann and Cosmin Pohoata for fruitful discussions.

\section{Proof of Theorem \ref{main}}\label{sect1}

For our argument it will be convenient to restate the problem in the following, essentially equivalent form. 
Fix $1 \le k \le d$. For a collection of vectors $v_1, \ldots, v_k \in \R^d$ let $\operatorname{Vol}_k(v_1, \ldots, v_k)$ be the $k$-dimensional volume of the parallelepided spanned by vectors $v_1, \ldots, v_k$. For brevity, we call $\operatorname{Vol}_k(v_1, \ldots, v_k)$ the determinant of the collection $v_1, \ldots, v_k$. Note that $\operatorname{Vol}_k(v_1, \ldots, v_k)$ equals to the square root of the determinant of the Gram matrix $(\langle v_i, v_j \rangle)_{i, j=1}^k$ of vectors $v_1, \ldots, v_k$.

We will use the following basic formula to estimate $\operatorname{Vol}_k(v_1, \ldots, v_k)$. We include a proof for completeness.
\begin{prop}\label{det}
Let $x_1, \ldots, x_{k} \in \R^d$ be arbitrary vectors, then for any $1\le \ell \le k$ we have
\begin{equation}\label{prod}
    \operatorname{Vol}_k(x_1, \ldots, x_k) = \operatorname{Vol}_\ell(x_1, \ldots, x_\ell) \operatorname{Vol}_{k-\ell}(x_{\ell+1}', \ldots, x_k'),
\end{equation}
where for $j=\ell+1,  \ldots, k$, we denote by $x_j'$ the projection of $x_j$ on the orthogonal complement to the space $V=\operatorname{span}(x_1, \ldots, x_\ell)$.
\end{prop}

\begin{proof}
View $x_1, \ldots, x_k$ as $d\times 1$ column vectors and consider the matrices $A= (x_1, \ldots, x_k)$ and $A' = (x_1, \ldots, x_\ell, x_{\ell+1}', \ldots, x_k')$. Note that $A' = A U$ for some upper triangular $k\times k$ matrix $U$ with all ones on the diagonal. The Gram matrices of the collections $x_1, \ldots, x_k$ and $x_1, \ldots, x_\ell, x_{\ell+1}', \ldots, x_k'$ are given by $G = A^T A$ and $G' = A'^T A'$, respectively. It follows that $\det G' = \det U^T G U = \det G$.
The Gram matrix $G'$ splits into $\ell\times \ell$ and $k-\ell \times k-\ell$ blocks which are the Gram matrices of the collections $x_1, \ldots, x_\ell$ and $x'_{\ell+1}, \ldots, x'_{k}$, respectively. This implies (\ref{prod}).
\end{proof}

Now for any $2\le k \le d+1$ define $\Delta'_{k, d}(n)$ to be the minimum $\Delta'$ such that any collection $X$ of $n$ unit vectors in $\R^{d+1}$ contains $k$ vectors with determinant at most $\Delta'$. 

\begin{prop}\label{reduction}
For any $2\le k \le d+1$ there are constants $c, c', C$ such that for all $n > k$:
$$
c \Delta_{k, d}(n) \le \Delta'_{k, d}(n) \le C\Delta_{k, d}(c' n).
$$
\end{prop}

\begin{proof}
Let $X \subset [0, 1]^d$ be an $n$-element set of points. Recall that for any points $x_1, \ldots, x_{k} \in \R^{d}$ the volume of the simplex spanned by them is equal to $\frac{1}{k!} \operatorname{Vol}_{k}((x_1, 1), \ldots, (x_{k}, 1))$.
Define a set $Y \subset S^d \subset \R^{d+1}$ as follows:
$$
Y = \left\{ \frac{(x, 1)}{\sqrt{|x|^2+1}},~x \in X \right\}.
$$
Note that the volumes of $(k-1)$-dimensional simplices in $X$ now precisely correspond to determinants of $k$-tuples of vectors $(x, 1)$, $x\in X$. Since for any $x \in [0, 1]^d$ we have $\sqrt{|x|^2+1} \in [1, \sqrt{d+1}]$, renormalization by these factors changes the determinants only by a constant factor. Thus, if $X$ does not contain $k$ points with volume at most $\Delta$ then $Y$ does not contain $k$ points with determinant at most $c \Delta$ for some $c > 0$ depending only on $d$ and $k$. This shows the first inequality.

Similarly, suppose that $Y \subset S^d$ has size $n$ and any $k$ points of $Y$ have determinant at least $\Delta'$. Consider the central projection of the sphere $S^d$ on the hyperplane $W = \{x ~|~ x_{d+1} = 1\}$. Then for some $c' > 0$ and a random rotation of $Y$ we get that at least $c' n$ points of $Y$ are projected in the cube $[0, 1]^d \times \{1\}$. Define $X$ to be the set of these points and note that $X$ does not contain $(k-1)$-simplices of volume $\Delta'/C$. This completes the proof of the second inequality.
\end{proof}

In light of Proposition \ref{reduction}, Theorem \ref{main} follows from an analogous statement about $\Delta'$.

\begin{lemma}
For any $2\le k \le d+1$, $1\le \ell <k$ and $n > k$ we have
$$
\Delta'_{k,d}(n) \le \Delta'_{\ell, d}(n) \Delta'_{k-\ell, d-\ell}(n-\ell).
$$
\end{lemma}

\begin{proof}
Let $X \subset S^d$ be an $n$-element set, we need to show that $X$ contains $k$ points with determinant at most $\Delta'_{\ell, d}(n) \Delta'_{k-\ell, d-\ell}(n-\ell)$. A small perturbation of points changes the determinants by a small amount, so we may assume that $X \subset S^d$ is a set of points in general position.

By the definition of $\Delta'_{\ell, d}$ there are points $x_1, \ldots, x_{\ell} \in X$ such that $\operatorname{Vol}_{\ell}(x_1, \ldots, x_{\ell}) \le \Delta'_{\ell, d}(n)$. 
Let $V = \operatorname{span}(x_1, \ldots, x_{\ell})$ and let $\pi$ be the projection on the orthogonal complement of $V$.
Consider the set $X' = \pi(X \setminus \{x_1, \ldots, x_{\ell}\})$ and define $Y = \{x / |x|, ~x \in X'\}$. By the general position assumption, $Y$ is a well-defined $n-\ell$-element set of points lying on a $d-\ell$-dimensional unit sphere $S^{d-\ell} \subset V^\perp$. 

So, by the definition of $\Delta_{k-\ell, d-\ell}(n-\ell)$, there exists a collection of points $y_1, \ldots, y_{k-\ell} \in Y$ such that $\operatorname{Vol}_{k-\ell}(y_1, \ldots, y_{k-\ell}) \le \Delta_{k-\ell, d-\ell}(n-\ell)$. Consider the preimages $z_1, \ldots, z_{k-\ell} \in X$ of these points. By Proposition \ref{det} we have
$$
\operatorname{Vol}_{k}(x_1, \ldots, x_{\ell}, z_1, \ldots, z_{k-\ell}) = \operatorname{Vol}_{\ell}(x_1, \ldots, x_{\ell})\operatorname{Vol}_{k-\ell}(\pi(z_1), \ldots, \pi(z_{k-\ell})) \le 
$$
$$
\le \operatorname{Vol}_{\ell}(x_1, \ldots, x_{\ell}) \operatorname{Vol}_{k-\ell}(y_1, \ldots, y_{k-\ell}) \le \Delta_{\ell,d}(n)\Delta_{k-\ell, d-\ell}(n-\ell),
$$
where in the first inequality we replaced vectors $\pi(z_j)$ with longer vectors $y_j = \pi(z_j)/ |\pi(z_j)|$ which increases the determinant. This completes the proof of the lemma, and by Proposition \ref{reduction}, of Theorem \ref{main}.
\end{proof}

\section{Proofs of corollaries}

For $2\le k\le d+1$ let us denote by $\delta_{k, d}$ the maximum number such that $\Delta_{k, d}(n) \ll n^{-\delta_{k, d}}$ holds for large $n$. We ignore the technical issues regarding the existence of the maximum and only use this notation for a cleaner presentation of the computations.

The bounds (\ref{ez}) and (\ref{lef}) imply that $\delta_{k, d} \ge \frac{k-1}{d}$ for all $2\le k\le d+1$ and $\delta_{k, d} \ge \frac{k-1}{d}+\frac{k-2}{2d(d-1)}$ for even $k$. Theorem \ref{main} now states that
\begin{equation}\label{rec}
\delta_{k, d} \ge \delta_{\ell, d} + \delta_{k-\ell, d-\ell},
\end{equation}
for all $2\le k\le d+1$ and $1\le \ell < k$, where we put $\delta_{1, d} = 0$ for convenience. Finally, the result of Cohen, Pohoata and the author implies that $\delta_{3, 2} \ge \frac87 + \frac{1}{2000}$.
We are now ready to perform the calculations for small $d$:
\begin{itemize}
    \item $\delta_{4, 3} \ge \delta_{2, 3} + \delta_{2, 1} \ge \frac13+1 = \frac43$,
    \item $\delta_{5, 4} \ge \delta_{3, 4} + \delta_{2, 1} \ge \frac24+1 = \frac32$,
    \item $\delta_{6, 5} \ge \delta_{4, 5} + \delta_{2, 1} \ge (\frac35+\frac{1}{20})+1 = \frac{33}{20}$,
    \item $\delta_{7, 6} \ge \delta_{4, 6} + \delta_{3, 2} \ge (\frac36 + \frac{1}{30}) + (\frac87+\frac{1}{2000}) \ge \frac{176}{105} > 1.676$. Note that all other ways to apply (\ref{rec}) lead to a slightly weaker exponent $5/3$: $\delta_{7, 6} \ge \delta_{2, 6} + \delta_{5, 4} \ge \frac16+\frac32 = \frac53$, $\delta_{7, 6} \ge \delta_{3, 6} + \delta_{4, 3} \ge \frac53$, and $\delta_{7, 6} \ge \delta_{5, 6} + \delta_{2, 1} \ge \frac53$,
    \item $\delta_{8, 7} \ge \delta_{2, 7} + \delta_{6, 5} \ge \frac17+\frac{33}{20} > 1.792$,
    \item $\delta_{9, 8} \ge \delta_{3, 8}+\delta_{6, 5} \ge \frac28+\frac{33}{20} = \frac{19}{10}$.
\end{itemize}
This proves Corollary \ref{cor2}. 

Using induction, we are going to show that for all $d\ge 3$ we have
\begin{equation}\label{log}
\delta_{d+1, d} \ge \log d - 6 + 10 d^{-\frac12}.
\end{equation}
This can be verified directly for $d \in \{3, \ldots, 8\}$.
The relation (\ref{rec}) applied with $l=1$ implies that the sequence $\delta_{d+1, d}$ is non-decreasing and so, in particular, $\delta_{d+1, d} \ge 1.9$ holds for all $d \ge 8$. This implies (\ref{log}) for all $3 \le d \le 100$.

Now let $d > 100$ be arbitrary and suppose that (\ref{log}) is true for all $3\le d' < d$. Let $1\le t \le d/2$ be an integer parameter which we will optimize later and apply (\ref{rec}) with $\ell=t$. Then by the induction assumption we get
$$
\delta_{d+1, d} \ge \delta_{t, d} + \delta_{d+1-t, d-t} \ge \frac{t-1}{d} + \log (d-t) - 6 + 10 (d-t)^{-\frac12},
$$
for any $x \in [0, 0.5]$ we have $\log(1-x) \ge -x-x^2$ and $(1-x)^{-\frac12} \ge 1+x/2$, so that we get 
$$
\log (d-t) \ge \log d - \frac{t}{d} - \frac{t^2}{d^2},
$$
$$
(d-t)^{-\frac12} = d^{-\frac12} \left(1-\frac{t}{d}\right)^{-\frac12} \ge d^{-\frac12} + \frac12 t d^{-\frac32}
$$
and so we conclude that
$$
\delta_{d+1, d} \ge \frac{t-1}{d} + \log d - \frac{t}{d} - \frac{t^2}{d^2} - 6 + 10 d^{-\frac12} + 5 t d^{-\frac32} = \log d - 6 + 10 d^{-\frac12} - \frac1d - \frac{t^2}{d^2} + 5t d^{-\frac32},
$$
note that for $d > 100$ we have $2 \sqrt{d} < d/2$ and let $t$ be an arbitrary integer in the interval $[\sqrt{d}, 2\sqrt{d}]$, then we get $-\frac{1}{d}-\frac{t^2}{d^2} + 5t d^{-\frac32} \ge 0$ and so $\delta_{d+1, d} \ge \log d -6 +10 d^{-\frac12}$. This completes the induction step and so (\ref{log}) is true for all $d \ge 3$.
By dismissing the additional term $10 d^{-\frac12}$ we conclude that $\delta_{d+1, d} \ge \log d -6$ for all $d\ge3$.

Now we prove Corollary \ref{cor15}. Let $2 \le k < d+1$ be arbitrary. By (\ref{rec}) applied with $\ell =2$,
$$
\delta_{k,d} \ge \delta_{k-2, d-2} + \delta_{2, d},
$$
iterating this until $k \ge 1$, gives
$$
\delta_{k,d} \ge \delta_{2, d} + \delta_{2, d-2} + \ldots + \delta_{2, d - 2 [k-1/2]} = \frac{1}{d} + \frac{1}{d-2} + \ldots + \frac{1}{d - 2 [\frac{k-1}{2}]} \ge 
$$
$$
\ge \frac{1}{2} \log \frac{d+2}{d-k+2},
$$
as claimed.

Now we prove Corollary \ref{cor3}. Let $d\ge 9$ and $3\sqrt{d} \le k \le d$. Apply (\ref{rec}) with $\ell = [k/2]$:
$$
\delta_{k, d} \ge \delta_{\ell, d} + \delta_{k - \ell, d-\ell} \ge \frac{\ell-1}{d} + \frac{k-\ell -1}{d-\ell} = \frac{k-1}{d} + \frac{(k-\ell)\ell - d}{d(d-\ell)} \ge  
$$
$$
\ge \frac{k-1}{d} + \frac{(k^2-1)/4 - d}{d^2} \ge \frac{k-1}{d} + \frac{k^2}{8d^2},
$$
as claimed.

\section{Simplices with faces of small volume}

To prove bounds on the Heilbronn functions $\Delta_{k, d}(n)$ with $k > d$ we will need the following lemma. Roughly speaking, the lemma states that any set of size $n$ in $[0,1]^d$ contains a simplex on $d+1$ vertices $S$, such that the volume of $S$ is small and, moreover, for all not too large $\ell$, all codimension $\ell$ faces of $S$ also have small volume. 

\begin{lemma}\label{useful}
    Let $a, t \ge 0$ and let $d=d_0 > d_1 > \ldots > d_t > 0$ be a sequence of positive integers with $d_{i}-d_{i+1} \ge a+1$ for all $i=0, \ldots, t-1$ and $d_t \ge a$.
    Let $X \subset [0, 1]^{d}$ be a set of $n$ points in general position. Then there exists a subset $Y \subset X$ of $d+1$ points with the following property. For any $0 \le \ell \le a$ and any subset $Y' \subset Y$ of size $d + 1 - \ell$ we have
    \begin{equation}\label{vol2}
    \operatorname{Vol}_{d-\ell}(\operatorname{conv}(Y')) \le C n^{- \gamma_{\ell}},    
    \end{equation}
    where
    \begin{equation}\label{gamma2}
    \gamma_{\ell} = \frac{d_0-d_1-1}{d_0} + \frac{d_1-d_2-1}{d_1} + \ldots + \frac{d_{t-1}-d_{t}-1}{d_{t-1}} + \frac{d_t - \ell}{d_t}.    
    \end{equation}
\end{lemma}

\begin{proof}
    The proof is by induction on $t$. First, let us reformulate the statement a little bit to make the induction cleaner. Instead of a subset $X \subset [0, 1]^d$ we consider an $n$-element set on the sphere $X \subset S^d$. We want to find a subset $Y \subset X$ of size $d+1$ such that $\operatorname{Vol}_{d+1-\ell}(Y') \lesssim n^{-\gamma_\ell}$ holds for any $\ell \le a$ and any $(d+1-\ell)$-element subset $Y'\subset Y$ (where by $\operatorname{Vol}_{d+1-\ell}$ we mean the function defined in Section \ref{sect1}). This is an equivalent problem (up to a constant factor) as any subset $X \subset [0,1]^d$ can be mapped to a subset on the sphere given by:
    $$
    X' = \left\{ \frac{(x,1)}{|(x,1)|},~ x \in S^d \right\},
    $$
    and any set on the sphere $X' \subset S^d$ contains a linearly sized subset which can be projected back into $[0,1]^d$. These operations preserve volumes of the convex hulls up to a constant factor, so it is sufficient to work with the spherical formulation of the problem.

     Let $t = 0$, and consider some $d=d_0 > 0$. Let $X \subset S^d$ be an $n$-element set. By pigeonhole principle, there exists a $(d+1)$-element subset $Y \subset X$ of diameter at most $c n^{-1/d}$ for some constant $c$ depending on $d$. Then for any $0\le \ell \le d$ any $(d+1-\ell)$-element subset $Y' \subset Y$ has volume at most $C n^{-\frac{d-\ell}{d}}$. Note that $d_0\ge a$ by assumption so we get the desired bound for all $0\le \ell \le a$.

    Now let $t\ge 1$ and consider a sequence $d= d_0 > \ldots > d_t > 0$ such that $d_i-d_{i+1} \ge a+1$ and $d_t \ge a$. Let $X \subset S^d$ be an $n$-element set in general position. By pigeonhole principle, there exists a $(d_0-d_1)$-element subset $Z \subset X$ of diameter at most $c n^{-1/d}$. Let $V = \operatorname{span}(Z)$ and let $\pi$ be the projection on the orthogonal complement of $V$.
    Consider the set $X' = \pi(X \setminus Z)$ and define $X'' = \{x / |x|, ~x \in X'\}$. By the general position assumption, $X''$ is an $n-(d_0-d_1)$-element set of points lying on a $d_1$-dimensional unit sphere $S^{d_1} \subset V^\perp$. Apply the induction assumption to the set $X''$ with parameters $a, t-1$ and the sequence $d_1 > \ldots > d_t$. We obtain a $(d_1+1)$-element set $Y' \subset X'' \subset S^{d_1}$ such that for any $0 \le \ell \le a$ and any subset $Y'' \subset Y'$ of size $d_1+1-\ell$ satisfies
    \begin{equation}\label{gamma}
    \operatorname{Vol}_{d_1+1-\ell}(Y'') \le C n^{-\gamma'_\ell}, 
    \end{equation}
    where 
    $$
    \gamma'_\ell = \frac{d_1-d_2-1}{d_1} + \frac{d_2-d_3-1}{d_2} +  \ldots + \frac{d_{t-1}-d_{t}-1}{d_{t-1}} + \frac{d_t - \ell}{d_t}.
    $$
    Let $Z'$ be the preimage of $Y'$ in $X$ (which is well-defined thanks to the general position assumption).
    Now define $Y = Z \cup Z'$, we claim that $Y$ satisfies the desired property. Indeed, let $0 \le \ell \le a$ be arbitrary and let $W \subset Y$ be an $(d+1-\ell)$-element subset. Write $W_1 = W \cap Z$ and $W_2 = W \cap Z'$. Note that by assumption $d_0-d_1 \ge a+1$ and so
    $$
    |W_1| = |W\cap Z| \ge |Z| - |Y \setminus W| \ge (d_0-d_1) - \ell \ge (a+1) - a > 0,
    $$
    i.e. $W_1$ is non-empty. By a similar computation, $W_2$ is non-empty as well. Denote $m = |W_1|$.
    Thus, by Proposition \ref{det} applied with $\ell = m$ we have
    $$
    \operatorname{Vol}_{d+1-\ell}(W) = \operatorname{Vol}_{m}(W_1) \operatorname{Vol}_{d+1-\ell -m}(W'_2),
    $$
    where $W_2' = \pi(W_2)$. Since $W_1 \subset Z$, the diameter of $W_1$ is at most $C n^{-1/d}$ and so $\operatorname{Vol}_{m}(W_1) \lesssim n^{-\frac{m-1}d}$.
    Let $W_2'' = \{w/|w|, ~w \in W_2'\}$ then $W_2''$ is a subset in $Y'$ of size $d + 1 - \ell - m$. Let 
    $$
    \ell' = |Y'| - |W_2''| = (d_1+1) - (d + 1 - \ell - m) = \ell + m +d_1 - d_0.
    $$
    Note that we have $m \le d_0-d_1$ and so $\ell' \le \ell \le a$ holds. 
    So by (\ref{gamma}) applied to $Y'' = W_2''$ we get
    $$
    \operatorname{Vol}_{d+1-\ell - m}(W'_2) \le \operatorname{Vol}_{d+1-\ell - m}(W''_2) \le C n^{-\gamma'_{\ell'}}.
    $$
    Combining the bounds above leads to an upper bound
    $$
    \operatorname{Vol}_{d+1-\ell}(W) \le C n^{- \frac{m-1}{d_0} - \gamma'_{\ell'}}.
    $$
    Expanding the number in the exponent gives
    \begin{align*}
    &\frac{m-1}{d_0} + \gamma'_{\ell'} = \frac{m-1}{d_0} + \gamma'_{\ell + m + d_1 - d_0} = \\
    &= \frac{m-1}{d_0} + \frac{d_1-d_2-1}{d_1} + \ldots + \frac{d_{t-1}-d_{t}-1}{d_{t-1}} + \frac{d_t - \ell - m - d_1 + d_0}{d_t} = \\
    &= \frac{m +d_1 -d_0}{d_0} + \frac{d_0-d_1-1}{d_0} + \frac{d_1-d_2-1}{d_1} + \ldots + \frac{d_{t-1}-d_{t}-1}{d_{t-1}} + \frac{d_t - \ell - m}{d_t} - \frac{m+d_1-d_0}{d_t}=\\
    &= \frac{m +d_1 -d_0}{d_0} + \gamma_\ell - \frac{m+d_1-d_0}{d_t} \ge \gamma_\ell
    \end{align*}
    where in the last line we used $m+d_1-d_0 \le 0$ and $d_0 > d_t$. We conclude that $\operatorname{Vol}_{d+1-\ell}(W) \le C n^{-\gamma_\ell}$ holds with $\gamma_\ell$ as in (\ref{gamma2}). This shows that the constructed set $Y$ satisfies the desired property and the induction step is complete. 
\end{proof}

\section{Proof of Theorem \ref{largek} and Proposition \ref{smalld}}

Let $a = k-d-1$, $t = [\frac{d}{a+1}]-1$ (note that $t\ge 0$ since $k \le 2d$) and consider the sequence of integers $d_i = d - i (a+1)$, for $i = 0, \ldots, t$. Let $X \subset [0,1]^d$ be a set of size $n$. 

Apply Lemma \ref{useful} to $X$ with parameters $a$, $t$ and the sequence $d_0 > \ldots > d_t$. We obtain a subset $Y\subset X$ of size $d+1$ such that for any $\ell \le a$ and any $Y'\subset Y$ of size $d+1-\ell$ we have
$$
\operatorname{Vol}_{d-\ell}(\operatorname{conv}(Y')) \lesssim n^{-\gamma_\ell},
$$
where
$$
\gamma_\ell = \frac{d_0-d_1-1}{d_0} +\frac{d_1-d_2-1}{d_1} + \ldots + \frac{d_{t-1}-d_t-1}{d_{t-1}} + \frac{d_t-\ell}{d_t}=
$$
$$
= \frac{a}{d_0} + \frac{a}{d_1} + \ldots + \frac{a}{d_{t-1}} + \frac{d_t - \ell}{d_t}.
$$
Using the bounds $d_i \le (t-i+2) (a+1)$ and $d_t \ge a+1$, leads to an estimate
$$
\gamma_\ell \ge \frac{a}{(t+2)(a+1)} + \frac{a}{ (t+1)(a+1)} + \ldots + \frac{a}{3(a+1)} + 1 - \frac{\ell}{a+1} \ge \frac{a}{a+1} \log \frac{t+3}{3} + 1 -\frac{\ell}{a+1}.
$$

Let $Z \subset X\setminus Y$ be a set of size $k-d-1 = a$ and diameter $\lesssim n^{-1/d}$. We claim that the set $S = Y \cup Z$ works. We will use the following simple volume estimate: 

\begin{prop}
    Given finite sets $A, B \subset [0,1]^d$ we have
$$
\operatorname{Vol}_d(\operatorname{conv}(A \cup B)) \lesssim \operatorname{Vol}_d(\operatorname{conv}(A))  + \operatorname{Vol}_d(\operatorname{conv}(B))+
$$
\begin{equation}\label{volume3}
 + \sum_{i=1}^{d} \sum_{\substack{A' \subset A, B' \subset B \\ |A'| = i, |B'| = d+1-i}} \operatorname{Vol}_{i-1}(\operatorname{conv}(A')) \operatorname{Vol}_{d-i}(\operatorname{conv}(B')).    
\end{equation}
\end{prop}

\begin{proof}
By the Caratheodory's theorem, the convex hull of the set $A\cup B$ is covered by the convex hulls of all its $(d+1)$-element subsets. So it is enough to prove (\ref{volume3}) when $|A| + |B| = d+1$. If one of the two sets is empty then the first two terms in (\ref{volume3}) upper bound the left hand side.
In the remaining case we need to show that
    $$
    \operatorname{Vol}_d(\operatorname{conv}(A \cup B)) \lesssim \operatorname{Vol}_{|A|-1}(\operatorname{conv}(A)) \operatorname{Vol}_{|B|-1}(\operatorname{conv}(B)).
    $$
This quickly follows from Proposition \ref{det} after projecting the sets on the sphere $S^d$ and noting that the volumes change only by a constant factor.
\end{proof}

We can now upper bound the volume of the convex hull of $S$ as follows:
$$
\operatorname{Vol}_d(\operatorname{conv}(S)) \lesssim \operatorname{Vol}_d(\operatorname{conv}(Y)) + \sum_{i = 1}^{a} \sum_{\substack{Y'\subset Y, Z' \subset Z,\\ |Y'|= d-i+1, |Z'| = i}} \operatorname{Vol}_{d-i}(\operatorname{conv}(Y')) \operatorname{Vol}_{i-1}(\operatorname{conv}(Z')) \lesssim 
$$
$$
\lesssim n^{-\gamma_0} + \sum_{i=1}^{a} n^{-\gamma_{i} - \frac{i-1}{d}} \lesssim n^{-\gamma_a - \frac{a-1}{d}}\lesssim n^{- \frac{a}{a+1}\log\frac{t+3}{3} -1 + \frac{a}{a+1} - \frac{a-1}{d}} \lesssim n^{- (1-\frac{1}{a+1}) (\log \frac{d}{a} - 1.1)},
$$
(where we bounded $\log 3 \le 1.1$). 
Note that $|Z| = a \le d_0 -d_1-1 < d+1$, so we do not have the term $\operatorname{Vol}_d(\operatorname{conv}(Z))$ in the above expression.
Since $a = k-d-1$, this implies that $\Delta_{k, d}(n) \lesssim n^{- (1-\frac{1}{k-d-1}) (\log \frac{d}{k-d-1} - 1.1)}$ holds. This proves Theorem \ref{largek}.

Now we prove Proposition \ref{smalld}. Let $a=1, t=1$ and consider the sequence $d_0 = 7 > d_1 = 3$. Repeating the argument above with these parameters produces a $9$-element set $S \subset [0,1]^7$ whose convex hull has volume at most $n^{- \gamma_1} = n^{- \frac{7-3-1}{7} - \frac{3-1}{3}} = n^{-\frac{23}{21}}.$


\begin{thebibliography}{20}

\bibitem{Bar} Barequet, Gill. ``A lower bound for Heilbronn's triangle problem in d dimensions.'' SIAM Journal on Discrete Mathematics 14, no. 2 (2001): 230-236.

\bibitem{BarN} Barequet, G., and J. Naor. ``Large k-D simplices in the d-dimensional unit cube.'' In Proc. 17th Canadian Conf. on Computational Geometry, pp. 30-33. 2006.

\bibitem{BHL} Bertram--Kretzberg, Claudia, Thomas Hofmeister, and Hanno Lefmann. ``An algorithm for Heilbronn's problem.'' SIAM Journal on Computing 30.2 (2000): 383-390.

\bibitem{B} Brass, Peter. ``An upper bound for the d-dimensional analogue of Heilbronn's triangle problem.'' SIAM Journal on Discrete Mathematics 19.1 (2005): 192-195.

\bibitem{CPZ} Cohen, Alex, Cosmin Pohoata, and Dmitrii Zakharov. ``A new upper bound for the Heilbronn triangle problem.'' arXiv preprint arXiv:2305.18253 (2023).

\bibitem{KPS1} Komlós, János, János Pintz, and Endre Szemerédi. ``On Heilbronn's triangle problem.'' Journal of the London Mathematical Society 2.3 (1981): 385-396.

\bibitem{KPS2} Komlós, János, János Pintz, and Endre Szemerédi. ``A lower bound for Heilbronn's problem.'' Journal of the London Mathematical Society 2.1 (1982): 13-24.

\bibitem{L} Lefmann, Hanno. ``Distributions of points in d dimensions and large k-point simplices.'' Discrete \& Computational Geometry 40.3 (2008): 401-413.

\bibitem{L2} Lefmann, Hanno. ``Distributions of points in the unit square and large k-gons.'' European Journal of Combinatorics 29.4 (2008): 946-965.

\bibitem{L3} Lefmann, Hanno. ``Generalizations of Heilbronn’s triangle problem.'' European Journal of Combinatorics 30.7 (2009): 1686-1695.

\bibitem{R1} Roth, Klaus F. ``On a problem of Heilbronn.'' Journal of the London Mathematical Society 1.3 (1951): 198-204.

\bibitem{R2} Roth, K. F. ``On a problem of Heilbronn, II.'' Proceedings of the London Mathematical Society 3.2 (1972): 193-212.

\bibitem{R3} Roth, K. F. ``On a problem of Heilbronn, III.'' Proceedings of the London Mathematical Society 3.3 (1972): 543-549.

\bibitem{S} Schmidt, Wolfgang M. ``On a problem of Heilbronn.'' Journal of the London Mathematical Society 2.3 (1972): 545-550.

\end{thebibliography}
\end{document}